\documentclass[12pt]{amsart}
\usepackage{etex}
\usepackage{etoolbox}
\patchcmd{\thebibliography}{*}{}{}{}
\usepackage{amssymb}
\usepackage{cite}
\usepackage{booktabs}
\usepackage{url}
\usepackage{hyphenat}
\usepackage{mathtools}
\usepackage{pifont}
\usepackage[all,cmtip]{xy}
\usepackage{ifpdf}
\usepackage{enumitem}
\usepackage{xcolor}
\usepackage{graphicx}
\usepackage[lmargin=1in,rmargin=1in,tmargin=1in,bmargin=1in]{geometry}
\usepackage[color=blue!20!white,textsize=tiny]{todonotes}

\usepackage{mathrsfs}
\usepackage{xr-hyper}

\definecolor{darkblue}{rgb}{0,0,0.4} 
\usepackage[colorlinks=true, citecolor=darkblue, filecolor=darkblue, linkcolor=darkblue,urlcolor=darkblue]{hyperref} 
\usepackage[all]{hypcap}

\allowdisplaybreaks

\usepackage{tikz}
\usetikzlibrary{matrix,arrows,3d,positioning,calc}
\tikzset{zxplane/.style={canvas is zx plane at y=#1,very thin}}
\tikzset{xyplane/.style={canvas is xy plane at z=#1,very thin}}
\tikzset{yzplane/.style={canvas is yz plane at x=#1,very thin}}
\tikzstyle{intext}=[rectangle,fill=white,inner sep=1pt,outer sep=2pt, fill opacity=0.7,text opacity=1]
\tikzset{doubar/.style={double, double equal sign distance, -implies}}

 \makeatletter
 \providecommand\@dotsep{5}
 \def\listtodoname{List of Todos}
 \def\listoftodos{\@starttoc{tdo}\listtodoname}
 \makeatother

\graphicspath{{draws/}{}}

\newcommand{\ZZ}{\mathbb Z}

\newcommand{\co}{\nobreak\mskip2mu\mathpunct{}\nonscript
  \mkern-\thinmuskip{:}\penalty300\mskip6muplus1mu\relax}

\newcommand{\lbracket}{[}
\newcommand{\rbracket}{]}

\DeclareMathOperator{\Hom}{Hom}

\theoremstyle{plain}

\newtheorem{proposition}[equation]{Proposition}
\newtheorem{lemma}[equation]{Lemma}
\newtheorem{corollary}[equation]{Corollary}

\newtheorem{definition}[equation]{Definition}

\theoremstyle{definition}

\theoremstyle{remark}

\newtheorem{remark}[equation]{Remark}

\newcommand\HH{\mathit{HH}}

\newcommand\Hochschild\HH
\newcommand\qCH{\mathit{qCH}}
\newcommand\qHH{\mathit{qHH}}

\newcommand\Id{\mathbb{I}}

\newcommand\DTP[1][]{\mathbin{\otimes^L_{#1}}}

\newcommand{\ModCat}{\mathsf{Mod}}

\makeatletter
\newcommand\honestalg[3]{\bigl\lbracket
\begin{smallmatrix} #1\@ifempty{#3}{}{&#3} \\ #2 \end{smallmatrix}
\bigr\rbracket}

\makeatother

\newcommand{\lsub}[2]{{}_{#1}#2}

\makeatletter

\newcommand*\wthelper[2]{%
        \hbox{\dimen@\accentfontxheight#1%
                \accentfontxheight#11.3\dimen@
                $\m@th#1\widetilde{#2}$%
                \accentfontxheight#1\dimen@
        }%
}
\newcommand*\accentfontxheight[1]{%
        \fontdimen5\ifx#1\displaystyle
                \textfont
        \else\ifx#1\textstyle
                \textfont
        \else\ifx#1\scriptstyle
                \scriptfont
        \else
                \scriptscriptfont
        \fi\fi\fi3
}
\makeatother

\makeatletter
\newlength\xvec@height%
\newlength\xvec@depth%
\newlength\xvec@width%
\newcommand{\xvec}[2][]{%
  \ifmmode%
    \settoheight{\xvec@height}{$#2$}%
    \settodepth{\xvec@depth}{$#2$}%
    \settowidth{\xvec@width}{$#2$}%
  \else%
    \settoheight{\xvec@height}{#2}%
    \settodepth{\xvec@depth}{#2}%
    \settowidth{\xvec@width}{#2}%
  \fi%
  \def\xvec@arg{#1}%
  \def\xvec@dd{:}%
  \def\xvec@d{.}%
  \raisebox{.2ex}{\raisebox{\xvec@height}{\rlap{%
    \kern.05em
    \begin{tikzpicture}[scale=1]
    \pgfsetroundcap
    \draw (.05em,0)--(\xvec@width-.05em,0);
    \draw (\xvec@width-.05em,0)--(\xvec@width-.15em, .075em);
    \draw (\xvec@width-.05em,0)--(\xvec@width-.15em,-.075em);
    \ifx\xvec@arg\xvec@d%
      \fill(\xvec@width*.45,.5ex) circle (.5pt);%
    \else\ifx\xvec@arg\xvec@dd%
      \fill(\xvec@width*.30,.5ex) circle (.5pt);%
      \fill(\xvec@width*.65,.5ex) circle (.5pt);%
    \fi\fi%
    \end{tikzpicture}%
  }}}%
  #2%
}
\makeatother

\begin{document}
\title{A remark on quantum Hochschild homology}

\author{Robert Lipshitz}
 \address{Department of Mathematics, University of Oregon\\
   Eugene, OR 97403}
\thanks{\texttt{RL was supported by NSF grant DMS-1810893.}}
\email{\href{mailto:lipshitz@uoregon.edu}{lipshitz@uoregon.edu}}

\date{\today}

\newcommand{\Tangles}{\mathsf{Tan}}
\newcommand{\AModCat}{\lsub{A}\ModCat}

\begin{abstract}
  Beliakova-Putyra-Wehrli studied various kinds of traces, in relation
  to annular Khovanov homology~\cite{BPW-Kh-HH}. In particular,
  to a graded algebra and a graded bimodule over it, they
  associate a quantum Hochschild homology of the algebra with
  coefficients in the bimodule, and use this to obtain a deformation
  of the annular Khovanov homology of a link. A spectral refinement of
  the resulting invariant was recently given by
  Akhmechet-Krushkal-Willis~\cite{AKW:qHH}.

  In this short note we observe that quantum Hochschild homology
  is a composition of two familiar operations, and give a short proof
  that it gives an invariant of annular links, in some
  generality. Much of this is implicit in~\cite{BPW-Kh-HH}.
\end{abstract}

\maketitle

\begin{definition}\cite[Section 3.8.5]{BPW-Kh-HH}
  Let $A$ be a graded ring, $M$ a chain complex of graded
  $A$-bimodules (so $M$ is bigraded), and $q\in A$ an invertible central element
  with grading $0$.  The \emph{quantum Hochschild complex} of $A$ with
  coefficients in $M$ and parameter $q$ has
  $\qCH_n(A;M)=M\otimes_\ZZ A^{\otimes_\ZZ n}$ and differential
  \begin{multline*}
    \partial(m\otimes a_1\otimes\cdots\otimes a_n)=
    ma_1\otimes a_2\otimes \cdots\otimes a_n
    +\sum_{i=1}^{n-1}(-1)^im\otimes a_1\otimes\cdots\otimes a_ia_{i+1}\otimes\cdots\otimes a_n
    \\+(-1)^nq^{-|a_n|}a_nm\otimes a_1\otimes\cdots\otimes a_{n-1}.
  \end{multline*}
  The homology of this complex is
  the \emph{quantum Hochschlid homology} $\qHH_\bullet(A;M)$ of $A$ with coefficients in $M$ and parameter $q$.
\end{definition}

The goal of this note is to reformulate this operation and deduce that
it often leads to annular link invariants. The data of $A$ and $q$ specifies a ring homomorphism $f_q\co A\to A$ defined on homogeneous elements $a$ of $A$ by 
\[
  f_q(a)=q^{-|a|}a,
\]
where $|a|$ denotes the grading of $a$. We can twist the left action
of the $A$-bimodule $M$ by $f_q$ to obtain a new bimodule
$\lsub{f_q}M$ which is equal to $M$ as a right $A$-module and has left
action given by the composition $A\otimes_\ZZ
\lsub{f_q}M\stackrel{f_q\otimes\Id}{\longrightarrow}A\otimes M\stackrel{m}{\longrightarrow}M=\lsub{f_q}M$. This operation is a special case of tensor product:
\[
  \lsub{f_q}M\cong \lsub{f_q}A\otimes_A M.
\]
Our first observation is:
\begin{proposition}
  The quantum Hochschild homology of $A$ with coefficients in $M$ is
  isomorphic to the ordinary Hochschild homology of $A$ with
  coefficients in $\lsub{f_q}M$.
\end{proposition}
\begin{proof}
  This is immediate from the definitions.
\end{proof}

Call a chain complex of graded $A$-bimodules $M$ \emph{weakly central} if for any
graded $A$-bimodule $N$ there is a quasi-isomorphism $M\DTP[A] N\simeq
N\DTP[A] M$.
\begin{lemma}\label{lem:fq-central}
  The bimodule $\lsub{f_q}A$ is weakly central.
\end{lemma}
\begin{proof}
  The isomorphism $M\otimes_A \lsub{f_q}A\to \lsub{f_q}A\otimes M$
  sends $m$ to $q^{-|m|}m$.
\end{proof}

We turn next to annular link invariants. Consider the category
$\Tangles$ with one object for each even integer and $\Hom(2m,2n)$
given by the set of isotopy classes of $(2m,2n)$-tangles (embedded in
$D^2\times [0,1]$). Given a (graded) algebra $A$, a \emph{very weak
  action} of $\Tangles$ on the category of $A$-modules is a choice of
quasi-isomorphism class of chain complex of (graded) $A$-bimodules
$C(T)$ for each $T\in\Hom(2m,2n)$ so that $C(T_2\circ T_1)$ is
quasi-isomorphic to $C(T_2)\DTP[A] C(T_1)$. For example, if we take
$A$ to be the direct sum of the Khovanov arc
algebras~\cite{Khovanov02:Tangles} then Khovanov defined a very weak
action of $\Tangles$ on $\AModCat$, and if we define $A$ to be the
direct sum of the Chen-Khovanov algebras~\cite{CK-kh-tangle} then
Chen-Khovanov defined a very weak action of $\Tangles$ on
$\AModCat$. (In fact, in both cases, they did more; cf.\
Remark~\ref{rem:functoriality}.)

Any $(2n,2n)$-tangle $T\subset D^2\times [0,1]$ has an \emph{annular closure} in $D^2\times S^1$.
\begin{proposition}\label{prop:vweak}
  Fix a very weak action of $\Tangles$ on $\AModCat$ and a weakly
  central $A$-bimodule $P$. Then for any $(2n,2n)$-tangle $T$, the
  isomorphism class of $\HH_*(A;C(T)\DTP[A] P)$ is an invariant of the
  annular closure of $T$.
\end{proposition}
\begin{proof}
  This is immediate from the definitions and the trace property of Hochschild
  homology, i.e., that given $A$-bimodules $M$ and $N$,
  \[
    \HH_*(A;M\DTP[A]N)\cong \HH_*(A;N\DTP[A]M).\qedhere
  \]
\end{proof}

\begin{corollary}
  Up to isomorphism, the quantum Hochschild homology of the
  Chen-Khovanov bimodule associated to a $(2n,2n)$-tangle $T$ is an
  invariant of the annular closure of $T$.
\end{corollary}
\begin{proof}
  This is immediate from Lemma~\ref{lem:fq-central},
  Proposition~\ref{prop:vweak}, and the fact that the Chen-Khovanov
  bimodules induce a very weak action of
  $\Tangles$~\cite{CK-kh-tangle}.
\end{proof}

\begin{remark}\label{rem:functoriality}
  To keep this note short, we have not discussed functoriality of
  these annular link invariants under annular cobordisms. To do so,
  one replaces $\Tangles$ by an appropriate $2$-category of tangles
  and weak centrality by a notion keeping track of the
  isomorphisms. See~\cite{BPW-Kh-HH} for further discussion.
\end{remark}

\bibliographystyle{hamsalpha}
\bibliography{heegaardfloer}
\end{document}
